\documentclass[leqno, final]{article}

\usepackage{ifdraft}
\ifdraft{\usepackage[a5paper, margin=5mm]{geometry}}{}
\usepackage{amsthm, amsmath, amssymb, mathtools}
\usepackage{thmtools, thm-restate}
\usepackage{enumitem}
\usepackage[disable=ifdraft]{microtype}
\usepackage{xcolor}
\usepackage{hyperref}
\hypersetup{
	breaklinks=true,   % splits links across lines
	pdfusetitle=true,  % \title and \author values into pdf metadata
}
\usepackage[maxbibnames=99, hyperref=true, backend=biber, style=numeric]{biblatex}
\DeclareNameAlias{default}{family-given}
\usepackage{cleveref}
\usepackage[pagewise]{lineno}

\newtheorem{corollary}{Corollary}
\newtheorem{lemma}{Lemma}
\newtheorem{proposition}{Proposition}
\declaretheorem[name=Theorem]{theorem}
\theoremstyle{definition}
\newtheorem{definition}{Definition}
\newtheorem{remark}{Remark}
\newtheorem{question}{Question}
\newtheorem{example}{Example}

\crefname{theorem}{Theorem}{theorem}
\crefname{proposition}{Proposition}{proposition}

\newcommand{\defeq}{\coloneq}
\newcommand{\la}{\langle}
\newcommand{\ra}{\rangle}

\renewcommand{\phi}{\varphi}

\newcommand{\uball}{{\bB_d}}

\newcommand{\Hbeta}{{\cH_\beta}}
\newcommand{\Hk}{{\cH_k}}

\newcommand{\Bd}{{\bB_d}}
\newcommand{\cB}{{\mathcal B}}

\newcommand{\cF}{{\mathcal F}}
\newcommand{\cG}{{\mathcal G}}
\newcommand{\cH}{{\mathcal H}}
\newcommand{\cL}{{\mathcal L}}
\newcommand{\cM}{{\mathcal M}}
\newcommand{\cX}{{\mathcal X}}
\newcommand{\bB}{{\mathbb B}}
\newcommand{\bC}{{\mathbb C}}
\newcommand{\bD}{{\mathbb D}}
\newcommand{\bN}{{\mathbb N}}
\newcommand{\bR}{{\mathbb R}}
\DeclareMathOperator{\Mult}{Mult}
\DeclareMathOperator{\Ran}{Ran}
\DeclareMathOperator{\Ker}{Ker}
\DeclareMathOperator{\Span}{span}
\DeclareMathOperator{\Rank}{rank}
\DeclareMathOperator{\Hol}{Hol}

\title{Generalized de Branges-Rovnyak spaces}
\author{Alexandru Aleman and Frej Dahlin}

\begin{document}
\maketitle
\begin{abstract}
Given the reproducing kernel $k$ of the Hilbert space $\Hk$
we study spaces $\Hk(b)$ whose reproducing kernel has the form
$k(1-bb^*)$, where $b$ is a row-contraction on $\Hk$. In terms
of reproducing kernels this is the most far-reaching generalization
of the classical de Branges-Rovnyaks spaces, as well as their very
recent generalization to several variables. This includes the
so called sub-Bergman spaces \cite{zhu_subbergm1} in one or several 
variables. We study some general properties of $\Hk(b)$ e.g. when the 
inclusion map into $\Hk$ is compact. Our main result provides a model for $\Hk(b)$
reminiscent of the Sz.-Nagy-Foia\c{s} model for contractions 
(see also \cite{aleman_conbackshift}). As an application we obtain 
sufficient conditions for the containment and density of the linear
span of $\{k_w:w\in\cX\}$ in $\Hk(b)$. In the standard cases this
reduces to containment and density of polynomials. These methods
resolve a very recent conjecture \cite{gu_hypercontractions}
regarding polynomial approximation in spaces with kernel
$\frac{(1-b(z)b(w)^*)^m}{(1-z\overline w)^\beta}, 1\leq m<\beta, m\in\bN$.
\end{abstract}
\newpage
%\tableofcontents

\newpage
\section{Introduction}
Motivated by models for contractions, de Branges and Rovnyak initiated 
in the 60's \cite{dbr65, dbr66} the study of the Hilbert spaces with a reproducing 
kernel of the form \begin{equation}\label{kb1}
	k(z,w) = \frac{1-b(z)b(w)^*}{1-z\overline w},\quad z,w\in \bD,
\end{equation}
Here, as usual, $\bD$ denotes the unit disc in the complex plane and the 
operator-valued function $b:\bD\to \cB(\bC,l^2)$ is analytic with operator 
norm bounded by one in $\bD$. These ideas have been considerably developed by 
Sarason \cite{sarason} who studied extensively these spaces in the case when 
$b$ is a scalar-valued analytic function, and called them 
\emph{sub-Hardy spaces} $\cH(b)$. For example, he represented them as ranges 
of the defect operators $(I-T_bT_b^*)^{1/2}$ endowed with the norm making 
these operators co-isometric. Here $T_b$ denotes the (analytic) Toeplitz 
operator with symbol $b$ on the Hardy space $H^2$.
\\
$\cH(b)$-spaces attracted a lot of attention during the years, a good 
account  on those developments can be found in the  book by Fricain and 
Mashreghi \cite{fricain-mashreghi} and in the references therein. An 
important idea of Sarason \cite{sarason} reveals a kind of dichotomy 
regarding the structure of these spaces. Roughly speaking, it turns out that 
the properties of $\cH(b)$ are either similar to those of  a shift invariant 
subspace of $H^2$, or  to those of the orthogonal complement of such a 
subspace (model space). The first alternative occurs when $b$ is not an 
extreme point of the unit ball of $H^\infty$, the algebra of bounded analytic 
functions in $\bD$, while the second occurs  when $b$ is an extreme point of 
that set. In fact, it is proved in \cite{sarason} that the following are 
equivalent: \begin{enumerate}[label=(\roman*)]
\item	$b\in\cH(b)$,
\item	$\cH(b)$ is invariant under the shift operator defined by $Sf(z) = zf(z)$,
\item	the polynomials are contained in $\cH(b)$ and are dense in the space,
\item	$b$ is not an extreme point of the unit ball of $H^\infty$.
\end{enumerate}
We shall refer to the above situation as Sarason's dichotomy.
\\
Zhu \cite{zhu_subbergm1}, \cite{zhu_subbergm2} developed these ideas by 
considering Toeplitz operators $T_b$ acting  on the Bergman space, i.e. the 
range of $(I-T_bT_b^*)^{1/2}$ with the appropriate norm. We should point out 
that these ranges are quite different. For example, it turns out that
$\Ran(I-T_bT_b^*) = \Ran(I-T_b^*T_b)$, which fails dramatically  in the Hardy 
space context.
\\
Such spaces are called \emph{sub-Bergman} spaces and by analogy to 
\eqref{kb1} they have the reproducing kernel \begin{equation}\label{sub_B}
	k(z,w)=\frac{1-b(z)\overline{b(w)}}{(1-z\overline w)^2},
\end{equation}
i.e. it is obtained by replacing in \eqref{kb1} the Szeg\H o kernel by the 
Bergman kernel.  All of these spaces are invariant for the forward shift and 
contain the polynomials. For our purposes it is important to note that in 
\cite{chu_density} Chu proved the density of polynomials in all sub-Bergman 
spaces.
\\
This idea generated a lot of further work
\cite{zhu_subbergm2,zhu_subbergm3,abkar-jafarzadeh,symesak} mainly devoted to 
extensions to weighted Bergman spaces in one or several variables.

Another important and very recent development regards the generalization of the
original de Branges-Rovnyak spaces to the context of several complex 
variables. From an operator-theoretic point of view, the most natural 
generalization of the Szeg\H o kernel to the unit ball in several complex 
variables is  the Drury-Arveson kernel \[
	k(z,w)=\frac1{1-\la z, w\ra}, \quad z,w\in \bB_d.
\]
Jury and Martin considered in \cite{jury-martin_extremal} 
reproducing kernel Hilbert spaces with kernel of the form
\begin{equation}\label{kb2} 
	k_b(z,w)=\frac{1-b(z)b(w)^*}{1-\la z, w\ra},
\end{equation}
where $b$ is a row-vector pointwise multiplier of the Drury-Arveson space. 
They characterized the multipliers $b$ for which the equivalent assertions
(i), or (ii) from above hold, showing in particular, that the Sarason dichotomy 
does occur in this case. Shortly after, Hartz \cite{hartz_columnrow} 
completed their work by showing that the equivalent statements (i) and (ii) 
above hold precisely when $b$ is an extreme point of the unit ball of the 
multiplier space. However, the density of polynomials (assertion (iii))  when 
(i) or (ii) hold, remains open in this context.
\\
Spaces with kernel of the form \eqref{kb2} appear also in the recent doctoral 
thesis of Sautel as models for certain commuting tuples which are 
expansive on the complement of  their null-space.  Moreover, 
in \cite{sautel} they provide sufficient conditions under which a 
kernel of the form \eqref{kb2} is a complete Nevanlinna-Pick kernel.

The present paper is concerned with a very general construction of this type. 
More precisely, we consider a reproducing kernel Hilbert space with 
reproducing kernel $k$ on a non-void set $\cX$, denoted by $\cH_k$. In order 
to deal with vector-valued $b$'s as well, we  assume that $b$ is a 
non-constant contractive multiplier of $\cH_k\otimes \ell^2$ into  $\cH_k$. 
This implies (see \S 2.1 below) that \[
	k^b(x,y)=(1-b(x)b(y)^*)k(x,y),\quad x,y\in \cX,
\]
is the reproducing kernel of a Hilbert function space on $\cX$ which we 
denote by $\cH_k(b)$. It could be called a \emph{sub-$\cH_k$ space}.
\\ 
Adopting this general point of view reveals  that a number of properties of 
the particular spaces mentioned above, especially in the sub-weighted Bergman 
case, only depend on general results about reproducing kernels. For example, 
in \S 3.2 we consider the embedding of $\cH_k(b)$ into $\cH_k$ which is 
always contractive.  We prove that a necessary condition for the compactness 
of this embedding is that $\|b(x_n)\|\to 1$, whenever $k(x_n,x_n)\to\infty$. 
\\
In the case when $\cH_k$ is a space of analytic functions on the unit disc
$\bD$ with $k(z_n,z_n)\to\infty$ whenever $|z_n|\to 1$ and $b$ is a
scalar-valued multiplier, the above condition implies that $b$ is a finite 
Blaschke product. In Theorem \ref{compact_embedd} we show that this condition 
is also sufficient provided that it is sufficient in the simplest case 
$b(z)=z$. This is a fairly far-reaching generalization of results of Zhu 
\cite{zhu_subbergm1, zhu_subbergm3} proved for standard weighted 
Bergman spaces.
\\
The results extend also to the context of several complex 
variables, namely when  $k$ is a power of the Drury-Arveson kernel and $b$ is 
an automorphism of the unit ball. We also supply some examples showing that, 
in general, the above condition fails to  be sufficient.

Beyond extending  these interesting  results, the study of general $\cH_k(b)$-
spaces has additional motivations.  In this paper we shall focus on the 
following two aspects: \begin{enumerate}[label=\arabic*)]
\item Reproducing kernel structure induced by operator  inequalities,
\item Reproducing kernels with a complete Nevanlinna-Pick factor.
\end{enumerate}
\noindent 1) When $b(0) = 0$ the original de Branges-Rovnyak spaces $\cH(b)$ 
with $b$ scalar-valued are precisely the spaces of analytic functions on $\bD$ 
with reproducing kernel normalized at $0$ where the backward shift $L$ is 
contractive and, in addition,  $I-LL^*$ has rank one. The recent paper 
\cite{aleman_conbackshift} provides a thorough study of such spaces without 
the last assumption, which yields Hilbert spaces with reproducing kernel of 
the general form \eqref{kb1}. Moreover, it turns out  that the case when 
$\Rank(I-LL^*) <\infty$ leads to a theory analogous to the one of sub-Hardy 
spaces, in particular Sarason's dichotomy holds.  Extending these ideas to 
the context of several variables is obviously very complicated due to the 
intricate structure of left inverses to the row operator given by 
multiplication of the coordinates. For this reason, the expansivity of that 
row operator plays the central role in Sautel's work \cite{sautel} mentioned above.
\\
The second hereditary inequality is more involved and was derived by Shimorin 
in \cite{shimorin_wold}. If $S$ denote the forward shift on a space of 
analytic functions $\cH$ on the disc with normalized reproducing kernel at 
the origin, Shimorin's  condition reads \begin{equation}\label{Shimorin_condition}
	\|Sf + g\|^2 \leq 2(\|f\|^2 + \|Sg\|^2).
\end{equation}
It turns out (see also \S3.1 below) that it characterizes general 
sub-Bergman spaces obtained by replacing the scalar function $b$ in 
\eqref{sub_B} by an operator-valued one. In \S 3.1 we prove a generalization 
of both inequalities. Assume that $\cH$ is a reproducing kernel Hilbert 
space of analytic functions in $\bD$ with normalized kernel at the origin.
We show that the inequality \[
	\left\|f_0+\sum_{n\ge 1}u_nf_n\right\|^2
	\le\|\varphi f_0\|^2+\sum_{n\ge 1}\|f_n\|^2, \quad f_n\in \cH,\, n\ge 0,
\]
where $\varphi, u_n,~n\ge 1$ are fixed multipliers, 
characterizes $\cH_k(b)$-spaces, with  $k$ a kernel of Bergman-type, a class 
introduced by McCullough and Richter in \cite{MR}.
\\
2) By a normalized complete Nevanlinna-Pick (CNP) kernel on the non-void set
$\cX$, normalized  at $z_0\in \cX$ we mean a kernel of the form \[
	(z,w)\mapsto\frac1{1-b(z)b(w)^*}, 
	\qquad b(z)b(w)^* = \la b(z), b(w)\ra_{\ell^2},\quad z,w\in \cX,
\]
where $\|b(z)\|_{l^2}\le 1,~z\in \cX,~b(z_0)=0$. Such reproducing kernels  
arise from solving certain interpolation problems for multipliers 
\cite{agler-mccarthy}. These kernels have a number of remarkable properties  
and have attracted a lot of attention recently. Here we are interested in a 
much larger class, namely those which have a normalized CNP factor, i.e. the 
reproducing kernels $k$ for  which there is a normalized CNP kernel $s$ such 
that $k/s$ is a reproducing kernel as well (positive definite). Such kernels
play a crucial role for the very general commutant lifting theorem by S. Shimorin
\cite{shimorin_leech}, and very recently in \cite{aleman_cnpfactor} it has 
been shown that functions in spaces with such kernels can be represented as 
quotients of multipliers. Given $k$ with the CNP factor 
$\frac1{1-b(z)b(w)^*}$  the "quotient" kernel \[ 
	k_b(z,w)=k(z,w)(1-b(z)b(w)^*),
\]
corresponds to the $\cH_k(b)$ space which actually  appears in Shimorin's 
commutant lifting theorem as well. One of the main difficulties in dealing 
with such spaces is to understand which functions belong to them  other than 
the reproducing kernels. In \S 4.1 we give sufficient conditions such that 
$\cH_k(b)$ contains the kernels $k_y\in \cH_k$ or even more, such that the span 
of these kernels is dense in $\cH_k(b)$ (Theorem \ref{thm:Hkb-k-dense}).
The proof of this approximation theorem is based on an analogue of the
Sz.-Nagy-Foia\c{s} model inspired by the work in \cite{aleman_conbackshift}. 
The idea is interesting in its own right and is presented in \S 3.3.
\\
There is a large class of examples satisfying these conditions, like the 
sub-Hardy spaces $\cH(b)$ with $b$ non-extremal in the unit ball of 
$H^\infty$ and all $\cH_k(b)$ spaces with $\cH_k$, essentially every weighted 
Bergman space on a domain in $\bC^d$ and $b$ a non-constant analytic row 
contraction on that domain.  Intuitively speaking, the kernels 
$k_y\in \cH_k,~y\in \cX$ play the same role as polynomials in the theory of 
spaces of analytic functions on the $d$-dimensional unit ball. In particular, 
these results extend Chu's density theorem \cite{chu_density} to sub-Bergman 
spaces on very general domains in $\bC^d$. In fact, for the standard weighted 
Bergman spaces on the ball we show in \S 5.1 that Sarason's dichotomy does 
not occur, that is conditions (i), (ii) and (iii) hold in any of such sub-
weighted Bergman spaces. 

Finally, these results and ideas  apply to an interesting class of kernels 
considered recently in \cite{gu_hypercontractions},  namely  positive
integer powers of the sub-weighted Bergman kernels \[
	\frac{(1-b(z)b(w)^*)^m}{(1-z\overline w)^\beta} \qquad m\in\bN, m<\beta\in\bR.
\]
These kernels are naturally connected to $m$-hypercontractive operators on 
standard  weighted Bergman spaces. In \cite{gu_hypercontractions} it was 
conjectured that polynomials are dense in the corresponding spaces of 
analytic functions on $\bD$. As an application of the results mentioned above 
we prove that the conjecture holds true. Our methods yield a proof in the 
higher dimensional case as well.

The paper is organized as follows. Section 2 serves to a preliminary purpose.
Section 3 contains general properties of $\Hk(b)$, in \S3.1 we discuss the
operator inequalities mentioned before. \S3.2 is devoted to the embedding
of $\Hk(b)$ in $\Hk$ and the cases where it is compact. The analogue of
the Sz.-Nagy-Foia\c{s} model is described in \S3.3.  Section 4 contains the 
general approximation results, while in Section 5 we apply these ideas to 
general weighted Bergman spaces. The proof of the conjecture in 
\cite{gu_hypercontractions} is deferred to section 6.

\section{Preliminaries}

\subsection{Reproducing kernels} For the sake of completeness we recall some 
basic properties of these objects. The material is standard and can be found in 
\cite{aronszajn} and \cite{paulsen}. If $\cX$ denotes a non-empty set, a 
function $k:\cX\times\cX\to\bC$ is called \emph{positive} 
(write $k\gg 0$) if for all finite subsets $\{x_1,\ldots,x_n\}\subset\cX$ the 
matrix $[k(x_i,x_j)]_{i,j=1}^n$ is positive semi-definite. Given  $k\gg 0$  
there exists a unique Hilbert space $\Hk$ consisting of complex-valued functions on $\cX$, 
such that point-evaluations are bounded linear functionals on $\Hk$ given by \[
	f(y)=\la f,k_y\ra_{\Hk},\quad x,y\in \cX,
\]
where we have used the common notation $k_y=k(\cdot,y)$. For this reason, $k$ 
is called a \emph{reproducing kernel}.

It is easy to verify that if $k,h$ are reproducing kernels on $\cX$ with 
$k-h\gg 0$, then $\cH_{h}$ is contractively contained  in $\Hk$. In fact, the 
map  $k_y\to h_y$ extends by linearity to a contraction, whose adjoint is the 
inclusion map. This yields the following result.
\begin{proposition}\label{Inclusion criterion}
Let $k$ be a reproducing kernel on $\cX$, then $f:\cX\to\bC$ is in $\cH_k$ if 
and only if \[
	c^2k(x,y) - f(x)\overline{f(y)} \gg 0
\]
for some $c>0$ and the least such $c$ is equal to $\|f\|_{\cH_k}$.
\end{proposition}

The classical Schur's product theorem  for positive semi-definite matrices 
applies to reproducing kernels and gives the following result.
\begin{theorem}[Schur's product theorem]
Let $k$ and $l$ be  reproducing kernels on $\cX$. Then $kl$ is also a 
reproducing kernel.
\end{theorem}
\begin{definition}
Let $k$ be a reproducing kernel on $\cX$. A function $\phi:\cX\to\bC$ is
a \emph{multiplier} of $\cH_k$ if $\phi f\in\cH_k$ for every $f\in\cH_k$.
The algebra of these operators is denoted $\Mult(\Hk)$.
\end{definition}
Note that if $\phi\in \Mult(\Hk)$, then the induced linear map $M_\phi$ on 
$\cH_k, ~M_\phi f=\phi f$ is bounded by the closed graph theorem. Moreover, 
we have 
\begin{equation}\label{mult-adjoint}
	M_\phi^*k_y = \overline{\phi(y)}k_y.
\end{equation}
Consequently \[
	|\phi(x)| \leq \|M_\phi\|, \qquad x\in\cX,\,k_x\neq0.
\]
Using also the fact 
that the linear span of $\{k_y:y\in\cX\}$ is dense in $\Hk$ one can derive 
the following result.
\begin{proposition}\label{Multiplier criterion}
Let $k$ be a reproducing kernel on $\cX$, then $\phi$ is a multiplier of 
$\cH_k$ if and only if \[
	k(x,y)(c^2-\phi(x)\overline{\phi(y)}) \gg 0
\]
for some $c>0$ and the least such $c$ is equal to $\|M_\phi\|$.
\end{proposition}
Given two sets $\cF,\cG$ of complex-valued functions on $\cX$ we shall 
consider the linear span of their pointwise products, that is, \[
	\Span\cF\cdot\cG = \Span\{fg:f\in\cF,g\in\cG\}.
\]
The next result is essentially a consequence of the above. We include a short 
proof for the convenience of the reader.
\begin{proposition}\label{prop:kernel-factored}
Let $s,t$ be reproducing kernels on $\cX$ and $k=st$. Then:
\begin{enumerate}[label=(\roman*)]
	\item	$\Span\Mult(\cH_s)\cdot\Mult(\cH_t)\subset\Mult(\cH_k)$,
	in particular $\Mult(\cH_s)\subset\Mult(\cH_k)$ contractively,
	\label{item:mult-inclusion}
	\item	If $f\in\cH_s$ and $g\in\cH_t$, then $fg\in\cH_k$ and
	$\|fg\|_{\cH_k}\leq\|f\|_{\cH_s}\|g\|_{\cH_t}$,
	in particular if the constant function $1\in\cH_t$, then
	$\cH_s$ is boundedly contained in $\cH_k$,
	\label{item:bound-containment}
	\item   If $\cF$ and $\cG$ are dense sets in $\cH_s$ and $\cH_t$, 
	respectively, then $\Span\cF\cdot\cG$ is dense in $\cH_k$.
	\label{item:weak-product-dense}
\end{enumerate}
\end{proposition}
\begin{proof}
\ref{item:mult-inclusion} Let $\phi\in\Mult(\cH_s)$ and $\psi\in\Mult(\cH_t)$,
then by Proposition \ref{Multiplier criterion} \begin{align*} 
	s(x,y)\|M_\phi\|^2\gg s(x,y)\phi(x)\overline{\phi(y)}&\gg0, \\
	t(x,y)\|M_\psi\|^2\gg t(x,y)\psi(x)\overline{\psi(y)}&\gg0.
\end{align*}
By Schur's product theorem, the order is preserved when multiplying the
left-hand side and the right-hand side, respectively. That is \[
	k(x,y)\|M_\phi\|^2\|M_\psi\|^2
	\gg k(x,y)\phi(x)\psi(x)\overline{\phi(y)\psi(y)}
\] 
and it follows that $\phi\psi\in\Mult(\cH_k)$.
Clearly, any element in $\Span\Mult(\cH_s)\cdot\Mult(\cH_t)$ is a sum of 
products as above. The contractive containment follows by letting above 
$\psi=1$  which gives  \[
	k(x,y)(\|M_\phi\|^2-\phi(x)\overline{\phi(y)})\gg0.
\]
\ref{item:bound-containment} is very similar and will be omitted.
\\\ref{item:weak-product-dense}
By \ref{item:bound-containment} we have that $\cF\cdot\cG\subset\cH_k$. If 
$y\in \cX,~\varepsilon>0$,  and $f\in \cF,~g\in \cG$ with \[
	\|f-s_y\|_{\cH_s} < \varepsilon, \quad \|g-t_y\|_{\cH_t} < \varepsilon,
\]
by the first part of  \ref{item:bound-containment} we have \begin{align*}
	\|fg-k_y\|_{\cH_k}
	&\le \|g\|_{\cH_t}\|f-s_y\|_{\cH_s} + \|s_y\|_{\cH_s}\|g-t_y\|_{\cH_t} \\
	&\le \epsilon(\|t_y\|_{\cH_t} + \|s_y\|_{\cH_s} + \epsilon).
\end{align*}
which gives the desired result by the density of the linear span of the kernel
functions $k_y,\,y\in\cX$ in $\cH_k$.	
\end{proof}
A repeated direct application of the last part of the proposition leads to 
the following result.
\begin{corollary}\label{prop:dense-algebra}
Let $\cF$ be an algebra of functions on $\cX$ and let $s^j,~j = 1,\ldots,n$ 
be reproducing kernels on $\cX$ such that $\cF$ is dense in 
$\cH_{s^j}, j = 1,\ldots,n$. If $k=\prod_{j=1}^ns^j$ then $\cF$ is dense in 
$\cH_k$.
\end{corollary}

We close the paragraph with a few words about the vector-valued case. The 
reproducing kernel of a Hilbert spaces of vector-valued functions is an 
operator valued function of two variables with similar properties to the 
scalar-valued version.
\\ 
In this paper we shall be concerned with a very special case.  Given a scalar 
reproducing kernel $k$ on $\cX$, we shall consider the Hilbert space  
$\Hk\otimes\ell^2$, that is, the space of functions $f:\cX\to\ell^2$ such that \[
	f = \begin{pmatrix} f_0 \\ f_1 \\ \vdots \end{pmatrix},\, f_j\in\Hk, j\geq0,
	\qquad \|f\|_{\Hk\otimes\ell^2}^2 = \sum_{j\geq0}\|f_j\|_\Hk^2 < \infty.
\]
This Hilbert space has the operator-valued reproducing kernel
$(x,y)\mapsto I_{\ell^2}k(x,y)$. In other words, given 
$f\in \Hk\otimes\ell^2,~y\in \cX,~e\in \ell^2$, we have \[
	\la f(y),e\ra_{\ell^2} = \la f,k_ye\ra_{\Hk\otimes\ell^2}.
\]
A multiplier $\phi$ from $\Hk\otimes\ell^2$ to $\Hk$ is row-vector valued 
function such that $\phi f=\sum_{j\ge0}\phi_jf_j\in\Hk$ for every $f\in\Hk\otimes\ell^2$. 
We write $\phi\in\Mult(\Hk\otimes\ell^2,\Hk)$. In this case \eqref{mult-adjoint}  
becomes \begin{equation}\label{eq:eigenvec}
	M_\phi^*k_y = \phi(y)^*k_y.
\end{equation} 
Note that $\phi(y)^*$ can be interpreted as a column-vector. Finally, let us 
also mention that Proposition \ref{Multiplier criterion} extends verbatim to 
this context by simply replacing $\overline{\phi(y)}$ by $\phi(y)^*$.

\subsection{Some unitarily invariant kernels on the unit ball of $\bC^d$}\label{sec:szego-powers}
Let $d$ be a positive integer and \[
	\bB_d=\{z\in\bC^d:\|z\|<1\},
\]
be the open unit ball in $\bC^d$. A kernel $k$ on $\bB_d$ is invariant under 
the unitary group on $\bC^d$ if and only if $k(z,w)$ is a function of 
$\la z,w\ra.$

An important class of examples, see \cite{aleman_besov}, are reproducing 
kernels of radially weighted Besov spaces \[
	B_\omega^\beta
	=\left\{f\in \Hol(\bB_d): ~\int_{\bB_d}|R^\beta f|^2\omega dV<\infty\right\}.
\]
Here $R^\beta$ is a power of the radial derivative operator 
$R=\sum_{j=1}^dz_j\frac{\partial}{\partial z_j}$, $V=V_d$ is the Lebesgue measure 
on the ball, and $\omega$ is a radial integrable function on $\bB_d$. It is 
easy to verify that monomials form an orthogonal basis in these spaces.
The particular case $\beta=\frac{d}{2},\,\omega=1$ corresponds to the
kernel on $\Bd$, which is defined by \[
	s(z,w) \defeq \frac1{1 - \la z,w\ra}, \qquad z,w\in\Bd.
\]
When $d=1$ it is the reproducing kernel for the Hardy space $H^2$ on the unit
disk $\bD$ and $s$ is called the Szeg\H o kernel. For $d>1$, the 
corresponding space is called the Drury-Arveson space and then $s$ is called the
Drury-Arveson kernel. This space plays a key role 
in the theory of commuting operator tuples. The recent survey 
\cite{hartz_invitation} contains  a lot of interesting material in this direction.\\
For any $\beta>0$ observe that \[
	s(z,w)^\beta = \sum_{n=0}^\infty a_n\la z,w\ra^n,
\]
where $a_n\geq0$ therefore by Schur's product theorem $s^\beta$ is a 
reproducing kernel.
\begin{definition}
Let $\beta>0$ we write $\Hbeta$ for the space of holomorphic functions on
$\Bd$ that corresponds to the reproducing kernel \[
	s^\beta(z,w) \defeq \frac1{(1-\la z,w\ra)^\beta}
\]
\end{definition}
These spaces are extensively studied in the literature. If $\beta>d$, then 
$\cH_{s^\beta}$ is a (standard) weighted Bergman space on $\Bd$ as noted in 
the next paragraph.

\subsection{Weighted Bergman spaces on a domain of $\bC^d$}\label{sec:weighted-bergman}
Let $\Omega$ be a domain in $\bC^d$ and let $\mu$ be a regular Borel  
positive measure on $\Omega$ such that the analytic functions in $L^2(\mu)$ 
form a closed subspace where point-evaluations are continuous linear 
functionals. These spaces are usually called weighted Bergman spaces and we 
denote them by $L^2_a(\mu)$. By assumption the space $L^2_a(\mu)$ has a 
reproducing kernel which we denote by $k^\mu$, and, as pointed out before,  
the norm is given by \[
	\|f\|^2 = \int_\Omega |f(z)|^2d\mu(z),\quad f\in L_a^2(\mu). 
\]
If $\beta>d$, the spaces $\Hbeta$ from the previous paragraph satisfy 
$\Hbeta=L^2_a(\mu)$ with $\Omega = \bB_d$, 
$d\mu=c_\alpha(1-|z|^2)^\alpha dV$, where $\alpha=\beta-(d+1)$, and 
$c_\alpha$ is a normalizing constant.

The vector valued version $L^2_a(\mu)\otimes\ell^2$ is defined with help of 
the norm \[
	\|f\|^2 = \int_\Omega\|f(z)\|_{\ell^2}^2d\mu(z) < \infty.
\]
The following simple observation plays an important role for the sequel. We 
denote throughout by $V$ the Lebesgue measure on $\bC^d$.
\begin{proposition}\label{equivalent-norm}
If $L_a^2(\mu)$ is a reproducing kernel closed subspace of $L^2(\mu)$ then 
there exists a measure $\tilde{\mu}$ such that 
$L_a^2(\tilde{\mu})\otimes \ell^2=L_a^2(\mu)\otimes \ell^2$ and for each 
compact subset $K$ of $\Omega$ there exists $c(K)>0$ with
\begin{equation}\label{same Bergman}
	\tilde{\mu}|K\ge c(K)V|K. 
\end{equation}\end{proposition} 
\begin{proof} 
Choose $v\in L^1(V), ~v>0$ a.e. with $\text{essinf } v|K>0$ for any  
compact subset $K$ of $\Omega$ and set \[
	d\tilde{\mu}=d\mu +\frac{v(z)}{1+\|k^\mu_z\|^2}dV.
\]
Since for every $f\in L_a^2(\mu)\otimes \ell^2$ and $z\in \Omega$ we have \[
	\frac{\|f(z)\|_{\ell^2}^2}{1+\|k^\mu_z\|^2}\le \|f\|,
\]
we obtain
$L_a^2(\tilde{\mu})\otimes \ell^2=L_a^2(\mu)\otimes \ell^2$ and that the 
norms are comparable.
\end{proof} 
 
One verifies that $\Mult(L^2_a(\mu))$ equals the algebra of bounded 
analytic functions on $\Omega$ with equality of norms. Similarly, 
$\Mult(L^2_a(\mu)\otimes\ell^2,L^2_a(\mu))$ consists of bounded analytic 
row-vectors and the multiplier norm coincides with the supremum norm. For our 
purposes it is important to note that these multipliers never attain their 
norm unless the symbols are constant. Indeed, if $\|bf\|=\|b\|_\infty\|f\|$ 
for some non-zero $f\in L^2_a(\mu)\otimes\ell^2$, then 
$\|b\|_\infty=\|b(z)\|_{\ell^2}$ for all $z\in \Omega$, which easily implies 
that $b$ is constant. In the terminology of \cite{aleman_freeouter} these 
Bergman spaces have no non-trivial \emph{sub-inner} multipliers. This is in 
strong contrast to smaller reproducing kernel spaces on the ball. In fact, 
more than that is true for measures  $\mu$ as above.
\begin{proposition}\label{pointev}
Let $b:\Omega\to \ell^2$ be analytic non-constant and such that $b(z)$ is a 
row-contraction for every $z\in \Omega$. If $\mu$ satisfies 
\eqref{same Bergman}, then for every $z\in \Omega$ there exists $c_z>0$ such 
that \[
	\|f\|^2-\|bf\|^2\ge c_z\|f(z)\|_{\ell^2}^2,
\]
for all $f\in L^2_a(\mu)\otimes\ell^2$.
\end{proposition}
\begin{proof} 
Let $z\in \Omega$ and $B_z$ be an open ball in $\bC^d$ centered at $z$  with 
$\overline{B_z}\subset \Omega$. By the assumption on $b$ there is $a_z>0$ such 
that $1-\|b(\zeta)\|_{\ell^2}^2 \ge a_z$, $\zeta\in B_z$. This implies for 
each $e\in \ell^2$, \[
	\|e\|_{\ell_2}^2-|b(\zeta)e|^2
	\ge (1-\|b(\zeta)\|_{\ell^2}^2)\|e\|_{\ell^2}^2
	\ge a_z\|e\|_{\ell^2}^2.
\]
Using also \eqref{same Bergman} we have for $f\in L^2_a(\mu)\otimes\ell^2$  
\begin{align*}
	\|f(z)\|_{\ell^2}^2
	&\le \frac1{V(B_z)}\int_{B_z}\|f\|_{\ell^2}dV
	\\
	&\le  \frac1{V(B_z)c(\overline{B_z})}\int_{B_z}\|f\|_{\ell^2}^2d\mu
	\\
	&\le  \frac1{V(B_z)c(\overline{B_z})a_z}\int_{B_z}(\|f\|_{\ell^2}^2-|bf|^2)d\mu 
	\\
	&\le\frac{\|f\|^2-\|bf\|^2}{V(B_z)c(\overline{B_z})a_z},
\end{align*}
which completes the proof.
\end{proof}

\section{Basic properties of generalized de Branges-Rovnyak spaces} 
Let us recall the general notion of a de Branges-Rovnyak space pointed out in 
the Introduction. Given a scalar-valued reproducing kernel $k$ on the 
non-empty set $\cX$ let $\cH_k$ denote the  corresponding Hilbert function space. 
\begin{definition}
Let $b\in\Mult(\cH_k\otimes\ell^2, \cH_k)$ with $\|M_b\|\leq1$. 
The sub-$\cH_k$ space, or the de Branges-Rovnyak space with respect to 
$\cH_k$ and $b$ is the Hilbert space corresponding to the reproducing kernel \[
	k^b(x,y) \defeq k(x,y)(1-b(x)b(y)^*),
	\qquad b(x)b(y)^* = \la b(x),b(y)\ra_{\ell^2},
\]
which is positive due to the multiplier criterion. We denote this space by 
$\cH_k(b)$.
\end{definition}
One can verify that as a set, $\cH_k(b)$ is equal to 
$\Ran(I-M_bM_b^*)^\frac12$ and its norm satisfies \[
	\la (I-M_bM_b^*)^\frac12f,(I-M_bM_b^*)^\frac12g\ra_{\cH_k(b)}
	= \la f,g\ra_{\cH_k},
\]
whenever $f,g\in\Ran(I-M_bM_b^*)^\frac12=\cH_k\ominus\ker(I-M_bM_b^*)^\frac12.$
Furthermore, $\cH_k(b)$ is contractively contained in $\cH_k$, this can be 
seen either from the above formula, or the fact that \[
	k(z,w) - k^b(z,w) = b(x)b(y)^*k(x,y),
\]
which is positive by Schur's product theorem.  

\subsection{Examples via operator inequalities}
It is well known that the backward shift \[
	Lf(z) = \frac{f(z)-f(0)}{z},
\]
is a contraction on any de  Branges-Rovnyak space $\cH(b)=\cH_s(b)$,  with 
$s(z,w)=\frac1{1-\overline{w}z}$ and $b(0)=0$.  The starting point of the 
work in \cite{aleman_conbackshift} is the observation that the converse holds 
true as well, as the following result shows.
\begin{proposition}[\cite{aleman_conbackshift}]\label{thm:backshift-inequality}
Let $\cH$ be a Hilbert space of analytic functions on $\bD$ with a reproducing 
kernel $k$ normalized at $0$. $L$ is contractive on $\cH$ if and only if there
exists an analytic row-contraction $b$ with $b(0)=0$ such that
\[
	k(z,w) = \frac{1-b(z)b(w)^*}{1-z\overline w}.
\]
\end{proposition}
For reproducing kernel Hilbert spaces of analytic functions in $\bD$ having 
the identity function $z$ as a multiplier and a reproducing kernel normalized 
at the origin,  the contractivity of $L$ is equivalent to the fact that the 
forward shift $M_z$ is expansive, that is, $\|zf\|\ge\|f\|$. In particular, 
such spaces have a general de Branges-Rovnyak kernel as above. 

There is a very recent extension of this idea to the context of several 
complex variables due to Sautel \cite{sautel}. It uses 
the Drury-Arveson kernel $s(z,w) = \frac1{1-\la z,w\ra}$
instead of the Szeg\"o kernel. The proof is much more involved in this case.
\begin{theorem}
Let $T=(T_1,\ldots,T_d):\cH^d\to\cH$ be a bounded row operator on an infinite
dimensional Hilbert space $\cH$, and let $m\in\bN$. T is unitarily equivalent
to $(M_z, \cH_s(b))$ for some $b\in\Mult(\cH_s\otimes\mathcal{D}, \cH_s\otimes\mathcal{E})$
and Hilbert spaces $\mathcal{D}$ and $\mathcal{E}$ such that $b(0) = 0$, 
$\dim\mathcal E = m$, and that $M_z$ is a bounded row operator if and only if $T$
satisfies the following four conditions:
\begin{enumerate}[label=(\roman*)]
	\item	$T_iT_j = T_jT_i$ for all $1\leq i,j\leq d$,
	\item	$\dim([\Ran(T-\lambda)]^\perp) = m$ for every $\lambda\in\Bd$,
	\item	$\|Tx\|_\cH\geq\|x\|_{\cH^d}$ for all $x\in\Ker(T)^\perp$,
	\item	$\bigcap\limits_{n\geq0}\sum\limits_{|w|=n}T_w\cH = \{0\}$.
\end{enumerate}
\end{theorem}

Sub-Bergman spaces are also related to an operator inequality which is 
considerably more subtle. It is due to S. Shimorin \cite{shimorin_wold} and 
it asserts that on the (unweighted) Bergman space on $\bD$, the forward shift 
$M_zf=zf$ satisfies 
\begin{equation}\label{shimorin-ineq}
\|zf+g\|^2\leq2(\|f\|^2+\|zg\|^2).
\end{equation}
As the previous inequalities, \eqref{shimorin-ineq} is \emph{hereditary} in 
the sense that it continues to hold for restrictions of the operator involved 
to its invariant subspaces. This estimate has far-reaching consequences 
regarding invariant subspaces of the Bergman shift. A crucial fact 
(see also \cite{AHR}) is that \eqref{shimorin-ineq}  is closely related to 
sub-Bergman kernels in the following way.
\begin{proposition}[Shimorin's inequality ]
Let $\cH=\Hk$ be a Hilbert space of analytic functions on $\bD$ such that 
the  reproducing kernel $k$ is normalized at $0$,  $M_z$ and $L$ are bounded 
on $\cH$. Then  the inequality \eqref{shimorin-ineq} holds on $\cH$ if and 
only if there exists an analytic row-contraction $b$ with $b(0)=0$ such that 
\begin{equation}\label{shimorin-kernel}
	k(z,w) = \frac{1-b(z)b(w)^*}{(1-z\overline w)^2},
\end{equation}
\end{proposition}
In order to study $M_z$-invariant subspaces of the Bergman space one needs to 
drop both assumptions in the above proposition, that is, to consider 
non-normalized kernels, as well as left inverses of the forward shift with a null 
space of arbitrary dimension.

Several interesting extensions of Shimorin's inequality have been found by Olofsson
and Wennman (see \cite{OW} and the references therein). 
\\
A  very general class of kernels containing all the above examples are the so-
called Bergman-type kernels introduced by McCullough and Richter in 
\cite{MR}. These are reproducing kernels in $\bD$ of the form 
\begin{equation}\label{bergman_type}
	k(z,w)=\frac1{1-\varphi(z)\overline{\varphi(w)}(1-u(z)u(w)^*)},\quad z,w\in \bD,
\end{equation}
with $\varphi(0)=0$ and $k(z,z)\to\infty$ when $|z|\to 1^-$. It is shown in 
\cite{MR} that $\rho(z)=\frac{\varphi(z)}{z}$ can be chosen to be outer 
and bounded from above and below. Clearly, the Szeg\"o kernel is obtained for 
$\varphi(z)=z, u=0$, and the unweighted Bergman kernel for 
$\varphi(z)=z\sqrt{2}, u(z)=\frac{z}{\sqrt{2}}$.
\\
Also, if $k$ is of Bergman-type we have $Mult(\cH(k))=H^\infty$  with 
equality of norms (see again \cite{MR}). The following theorem extends all 
previous results in one complex variable to the general context of 
Bergman-type kernels.

\begin{theorem}\label{bergman_type kernel} 
Let $k$ be a  Bergman-type kernel on $\bD$ of the form \eqref{bergman_type}. 
Let $\cH$ be a Hilbert space of analytic functions on $\bD$ whose reproducing 
kernel is normalized at $0$ and assume that $\varphi\in Mult(\cH)$, the range 
of  $M_\varphi$ is closed of co-dimension $1$, and that $  u\in Mult(\cH\otimes \ell^2)$. 
Then  $\cH=\cH_k(b)$ for an analytic row-contraction $b$ in $\bD$ with $b(0)=0$, if and only 
\begin{equation}\label{bergman_type inequality}
	\left\|f_0+\sum_{n\ge 1}u_n	f_n\right\|^2\le\|\varphi f_0\|^2+\sum_{n\ge 1}\|f_n\|^2,
\end{equation}
whenever $f_0, f_1\ldots, f_n,\ldots\in  \cH$.	
\end{theorem}
\begin{proof}
The assumptions on $M_\varphi$ imply that its range equals the space of 
functions in $\cH$ vanishing at the origin. Consequently, the operator 
$L_\varphi$ defined by \[
	L_\varphi f=\frac{f-f(0)}{\varphi},\quad f\in \cH,
\]
is a bounded left-inverse of $M_\varphi$. Then the inequality 
\eqref{bergman_type inequality} is equivalent to \begin{equation}\label{bergman_type inequality1}
	\left\|L_\varphi f_0+\sum_{n\ge 1}u_nf_n\right\|^2\le\|f_0\|^2+\sum_{n\ge 1}\|f_n\|^2.
\end{equation}
Indeed, if \eqref{bergman_type inequality} holds, we apply it to $L_\varphi f_0, f_1,\ldots$. 
From the fact that the kernel of $\cH$ is normalized at $0$ we obtain \[
	\|\varphi L_\varphi f_0\|^2=\|f_0-f_0(0)\|^2 \le \|f_0\|^2,
\] 
and \eqref{bergman_type inequality1} follows. 

Conversely, if \eqref{bergman_type inequality1} holds we apply it to 
$\varphi f_0, f_1,\ldots$ and obtain \eqref{bergman_type inequality}.
 Now \eqref{bergman_type inequality1} is further equivalent to the fact that 
the operator $T:\bigoplus_\mathbb{N}\cH\to \cH$ defined by \[
	T\sum_{n\ge 0}f_n=L_\varphi f_0+\sum_{n\ge 1}u_nf_n,
\]
is contractive. This can be re-written as \[
	I-TT^*\ge 0\,\Leftrightarrow (I-TT^*)h\gg 0  \,\Leftrightarrow (I-TT^*)h=bb^*,
\]
where $h$ is the reproducing kernel in $\cH$ and $b$ is an analytic row-contraction 
in $\bD$. Since \[
	L_\varphi L_\varphi^*h(z,w) 
	= \frac{h(z,w)-1}{\phi(z)\overline{\phi(w)}},\quad M_uM_u^*h(z,w)
	= u(z)u(w)^*h(z,w),\, u\in \Mult(\cH\otimes\ell^2, \cH),
\]
we can solve for $h$ in $(I-TT^*)h=bb^*$ and obtain after a straightforward 
calculation \[
	h(z,w)
	= \frac{1-\phi(z)\overline{\phi(w)}b(z)b(w)^*}{
		1-\phi(z)\overline{\phi(w)}(1-u(z)u(w)^*)}
	= k(z,w)(1-\phi(z)\overline{\phi(w)}b(z)b(w)^*)
\]
i.e. $h$ is the reproducing kernel of $\cH_k(\phi b)$.
\end{proof}

Besides the Szeg\H o, or the unweighted Bergman kernel, some standard examples 
of Bergman-type kernels are \[
	k_\beta(z,w) =\frac{1}{(1-z\overline{w})^\beta}, \quad z,w\in \bD,\, 1<\beta<2.
\]
These examples are discussed in \cite{OW}. The corresponding Hilbert spaces 
are standard weighted Bergman spaces. Of course, there are many other 
examples in this class, in particular in \cite{MR} there are produced 
examples such that the forward shift on the corresponding space is not 
subnormal.

\subsection{Compact embedding}
As mentioned in the beginning of this section $\cH_k(b)$ is always contained 
contractively in $\cH_k$. It is natural to ask when the contractive inclusion 
map is compact. As pointed out in the Introduction this question attracted
attention in the context of sub-weighted Bergman spaces.

If the original space $\cH_k$ is smaller, then the question is less interesting as 
we shall see below. We start with a simple necessary condition for the 
compactness of the inclusion map which holds under mild assumptions in the 
most general context of a kernel $k$ on the non-void set $\cX$.
\begin{proposition}\label{nec-compact-embedding} 
Let $k$ be a reproducing kernel on a non-empty set $\cX$ with the property 
that $k_y$ is a bounded function for every $y\in\cX$. Let 
$b\in \Mult(\cH_k\otimes \ell^2, \cH_k)$ with $\|M_b\|\leq1$ and suppose that the 
inclusion map from $\cH_k(b)$ to $\cH_k$ is compact. Then for every sequence
$(x_n)$ in $\cX$ with $\lim_{n\to\infty}k(x_n,x_n)=\infty$, we have \[
	\lim_{n\to\infty}\|b(x_n)\|_{\ell^2}=1.
\]
\end{proposition}
\begin{proof} Let us begin with the simple observation that if 
$J:\cH_k(b)\to\cH_k$ is the inclusion map,then its adjoint $J^*$ satisfies \[
	J^*k_y(x)=(1-b(x)b^*(y))k_y(x),\quad x,y\in \cX.
\]
Now let $(x_n)$ be a sequence as in the statement and set 
$f_n = \frac{k_{x_n}}{\|k_{x_n}\|}$, so that $\|f_n\|=1$. Furthermore, by the 
assumption on $k$ we have \[
	\la f_n, k_{y}\ra
	 = \frac{k(x_n,y)}{\|k_{x_n}\|}
	= \frac{\overline{k_y(x_n)}}{\|k_{x_n}\|}\to0 \text{ as $n\to\infty$}.
\]
This proves that $(f_n)$ tends weakly to $0$ in $\cH_k$, hence if $J$ is 
compact, so is $J^*$, that is, $J^*f_n$ must tend in norm to 0. Thus, \[
	\|J^*f_n\|^2 = 1-\|b(x_n)\|^2 \to0.
\]
and the result follows.
\end{proof}
The result is of particular interest for most common reproducing kernel 
Hilbert spaces of analytic functions in the unit disc, in the case when $b$ 
is scalar-valued. Indeed, if $\cH_k$ is such a space with \[
	k(z,z)\to \infty,\quad |z|\to 1^-,
\]
and $b\in Mult(\cH_k)$ is contractive, by \cref{nec-compact-embedding} the 
compactness of the inclusion $J:\cH_k(b)\to\cH_k$ implies that the bounded 
analytic function $b$ satisfies for all $\lambda\in \mathbb{T}$ \[
	\lim_{z\to\lambda}|b(z)|=1,
\]
hence $b$ is a finite Blaschke product. However, in this generality,  finite 
Blaschke products may fail to be contractive multipliers so the corresponding 
generalized de Branges-Rovnyak spaces might not exist.
\begin{example}
\phantom{}
\begin{enumerate}[label=\alph*)]
\item In the Dirichlet space $D$ with kernel \[
	-\frac1{z\overline{w}}\log(1-z\overline{w}),\quad z,w\in \bD,
\]
every finite Blaschke product $b$ is a multiplier, but (see \cite{RS}) \[
	\|bf\|>\|f\|,\quad f\in D\setminus\{0\}.
\]
\item Let $\cH_2$ be the unweighted Bergman space and $\cH_1=H^2$ be the 
Hardy space on the unit disc.
Consider the reproducing kernel Hilbert space $\cH$ consisting of analytic 
functions $f$ in $\bD$ which can be written (uniquely) as \[
	f(z) =u_1(z^2)+zu_2(z^2), \quad u_1\in \cH_1,\, u_2\in \cH_2,
\]
with norm \[
	\|f\|^2=\|u_1\|^2_{\cH_1}+\|u_2\|^2_{\cH_2}.
\]
Then obviously, $M_z$ is unbounded on $\cH$ while if $\varphi\in H^\infty$, the 
function $\psi(z)=\varphi(z^2)$ is a multiplier with $\|M_\psi\|=\|\psi\|_\infty.$
\end{enumerate}\end{example}
\begin{proposition}\label{contractive_Blaschke} 
Given a reproducing kernel Hilbert space $\cH_k$  of analytic functions in 
$\bD$, the finite Blaschke products are contractive multipliers if and only 
if the identity function $\zeta(z)=z$ is a contractive multiplier.
\end{proposition}
\begin{proof} One direction is obvious since the identity is a finite 
Blaschke product, while the converse follows by standard functional calculus 
for contractions from the equality $b(M_z)=M_b$ valid for any finite Blaschke 
product.
\end{proof}
A similar result can be derived for the compactness of the embedding of 
$\cH_k(b)$ into $\cH_k$, but the proof is somewhat more involved.  Our next 
result provides a far-reaching generalization of the result in  \cite{zhu_subbergm3}. 
\begin{theorem}\label{compact_embedd} Assume that $\cH_k$ is a reproducing 
kernel Hilbert space of analytic functions on $\bD$ such that \begin{enumerate}[label=(\roman*)]
\item $M_z$ is contractive on $\cH_k$,
\item $k(z,z)\to\infty$ as $|z|\to1^-$.
\end{enumerate}
Then if $b\in Mult(\cH_k)$ is contractive, the inclusion 
map From $\cH_k(b)$ into  $\cH_k$ is compact if and only if $I-M_zM_z^*$ is 
compact and $b$ is a finite Blaschke product.
\end{theorem}
The proof is essentially based on the  identity below. A similar formula appears
in \cite{gu_hypercontractions}. Given $a\in \bD$ we denote by $\varphi_a$ the 
disc-automorphism \[
	\varphi_a(z)=\frac{z-a}{1-\overline{a}z},\quad z\in \bD.
\]
\begin{lemma}\label{B_identity} 
Let $s(z,w)=\frac1{1-z\overline w}$ and let $b$ be a finite Blaschke product with 
zeros $a_1,\ldots,a_n\in \bD$. Then \begin{equation}\label{eq:blaschke-formula}
	(1-b(z)\overline{b(w)})s(z,w) 
	= \sum_{l=1}^n(\phi_{a_l}'(z)\overline{\phi_{a_l}'(w)})^\frac12
	\prod_{j=1}^{l-1}\phi_{a_j}(z)\overline{\phi_{a_j}(w)}.
\end{equation}\end{lemma}
\begin{proof} We proceed by induction. For $a\in\bD$ a direct computation 
gives \[
	(1-\phi_a(z)\overline{\phi_a(w)})s(z,w) 
	= (\phi_a'(z)\overline{\phi_a'(w)})^\frac12,
\]
which is the case $n=1$.  Suppose that \eqref{eq:blaschke-formula} holds for 
any Blaschke product with $n$ zeros, and let $b$ be the finite Blaschke 
product with zeros $a_1,\ldots, a_{n+1}\in \bD$. Let $b_0=1$ and let 
$b_k,~k\ge 1,$ denote the Blaschke product having only the first $k$ zeros, 
$a_1,\ldots,a_k$. Write \[
	1-b(z)\overline{b(w)}
	=1-b_n(z)\overline{b_n(w)} + b_n(z)\overline{b_n(w)}(1-\varphi_{a_{n+1}}(z)\overline{\varphi_{a_{n+1}}(w)}).
\]
By the induction hypotheses \[
	(1-b_n(z)\overline{b_n(w)})s(z,w)=\sum_{l=1}^n(\phi_{a_l}'(z)\overline{\phi_{a_l}'(w)})^\frac12 b_{l-1}(z)\overline{b_{l-1}(w)},
\]
and by the case $n=1$, \[
	b_n(z)\overline{b_n(w)}(1-\varphi_{a_{n+1}}(z)\overline{\varphi_{a_{n+1}}(w)})s(z,w)
	= (\phi_{a_{n+1}}'(z)\overline{\phi_{a_{n+1}}'(w)})^\frac12b_n(z)\overline{b_n(w)},
\]
which proves the induction step and the result follows.
\end{proof}
\noindent{\it Proof of \cref{compact_embedd}.}
The assumption $\|M_z\|\le 1$ implies that for all $a\in \bD$, the functions 
$u_a(z)=(1-\overline{a}z)^{-1},~z\in \bD$ are multipliers, in fact 
$M_{u_a}=(I-\overline{a}M_z)^{-1}$. Assume that $I-M_zM_z^*$ is compact and let 
$b$ be a finite Blaschke product  with zeros $a_1,\ldots,a_n\in \bD$.
Multiply both sides of \eqref{eq:blaschke-formula} by $k(z,w)/s(z,w)$, use 
the identity $u(z)\overline{u(w)}k(z,w)=M_uM_u^*k_w(z)$  together with the density 
of the linear span of $\{k_w:~w\in \bD\}$ to obtain an equality of the form 
\begin{equation}\label{op:Blaschke_equation}
	I-M_bM_b^*= \sum_{l=1}^nM_{v_l}(I-M_zM_z^*)M_{v_l}^*,
\end{equation}
with $v_1,\ldots,v_n\in Mult(\cH_k)$.  Thus, by assumption, each operator in 
the sum is compact, hence, so is $I-M_bM_b^*$. Now if $J:\cH_k(b)\to \cH_k$ 
is the inclusion map,  as above it follows that \[
	J^*k_w(z)=(1-b(z)\overline{b(w)})k(z,w),
\]
hence $JJ^*=I-M_bM_b^*$ is compact. Conversely, if $JJ^*=I-M_bM_b^*$ is 
compact, by \cref{nec-compact-embedding}, $b$ must be a Blaschke product, so 
that \eqref{op:Blaschke_equation} holds. Note that all operators on the right 
hand side are positive, and in the first term we have \[
	v_1(z)
	= (\varphi_{a_1}'(z))^{\frac12}
	= \frac{(1-|a_1|^2)^{\frac12}}{1-\overline{a_1}z},
\]
is an invertible multiplier. This leads to \[
	M_{v_1}^{-1}(I-M_bM_b^*)(M_{v_1}^{-1})^*\ge I-M_zM_z^*.
\]
Since the left hand side is compact and the right hand side is positive it 
follows that $I-M_zM_z^*$ is compact and the proof is complete. \qed

The theorem applies not only to arbitrary weighted Bergman spaces on $\bD$, 
or the Hardy space. The result holds in many other reproducing kernel spaces, 
for example those corresponding to radial kernels \[
	k(z,w)=\sum_{n\ge 0} k_n(z\overline{w})^n,\quad z,w\in \bD,
\]
with $0<k_n\le k_{n+1},~\lim_{n\to\infty}\frac{k_n}{k_{n+1}}=1.$
\\
It is an interesting question to characterize the compact embedding of 
$\cH_k(b)$ into $\cH_k$ for a a contractive row multiplier $b$.

The situation is much more complicated for analytic reproducing kernels on 
the unit ball $\bB_d,~d>1$. A particular case that can be treated with the 
method above is when $b$ is an automorphism of $\uball$ and $\cH_k=\cH_\beta$ 
for some $\beta>1$. In this case we have the following identity similar to 
\eqref{eq:blaschke-formula} (see for example \cite{rudin_uball}), \[
	1-\la b(z),b(w)\ra 
	= \frac{(1-\la a,a\ra)(1-\la z,w\ra)}{(1-\la z,a\ra)(1-\la a,w\ra)}
\]
where $a=b^{-1}(0)$. This leads to \[
	s^\beta(z,w)(1-b(z)b(w)^*)
	= s^{\beta-1}(z,w)\phi(z)\overline{\phi(w)},
\]
where $\phi(z)=\frac{(1-\la a,a\ra)^{\frac12}}{1-\la z, a\ra}$ is an 
invertible multiplier of $\cH_{\beta-1}$. In other words, 
$\cH_\beta(b)=\cH_{\beta-1}$  with equivalent norms. This proves the 
following result. \begin{proposition}
If $\beta>1$ and $b$ is an automorphism of $\uball$, then $\Hbeta(b)$ is 
compactly embedded in $\Hbeta$.
\end{proposition}

\subsection{A model for $\Hk(b)$} One of the  main difficulties that occur 
when dealing with de-Branges-Rovnyak spaces and their generalizations is 
finding functions in the space  other than finite linear combinations of 
reproducing kernels. In some cases this led to a number of recent non-trivial 
results. For example, the main result in \cite{aleman_conbackshift} shows 
that in one variable, $\cH_s(b)$ contains a dense subset of functions that 
extend continuously to $\overline{\bD}$. Further refinements of that result have 
been recently obtained by Limani and Malman \cite{LM}. The method in \cite{
aleman_conbackshift} relies on a special formula for the norm in $\cH_s(b)$ 
which turns out to be essentially related to the famous Nagy-Foia\c{s} model 
for the backward shift on these spaces (see \cite{NFBK}).
Some of these ideas do extend to the general context of $\cH_k(b)$-spaces and 
provide useful information about their structure. To illustrate this, 
consider a scalar-valued reproducing  kernel $k$ on the non-empty set $\cX$ 
with corresponding Hilbert space $\cH_k$ and fix a contractive row-multiplier 
$b\in\Mult(\cH_k\otimes\ell^2, \cH_k)$. The considerations below apply to a 
scalar, or finite dimensional vector $b$ in a trivial way. Therefore for the 
rest of the subsection we shall consider the case when 
$b\in\Mult(\cH_k\otimes\ell^2, \cH_k)$.

Denote by $\Delta=(I-M_b^*M_b)^\frac12:\cH_k\otimes \ell^2\to \cH_k\otimes \ell^2$. 
If $\Delta$ (or $I-M^*_bM_b$) is injective, then \[
	\|f\|_\Delta = \|\Delta f\|, 
\]
is a norm on $\cH_k\otimes \ell^2$. In this case we let $\cL_\Delta$ be the 
completion of $\cH_k\otimes \ell^2$ with respect to $\|\cdot\|_\Delta$. We 
shall be interested in the closed subspace  $\cM\subset \cH_k\oplus\cL_\Delta$ 
defined by \begin{equation}\label{M_def}
	\cM=\{(bf,f):f\in\cH_k\otimes\ell^2\}
\end{equation}
Let us record two simple facts about this subspace. 
\begin{proposition}\label{M_perp}
Suppose that $\Delta$ is injective and let $\cM\subset \cH_k\oplus\cL_\Delta$ 
be the subspace given by \eqref{M_def}. Then:\begin{enumerate}[label=(\roman*)]
\item $\cM$ is closed,
\item For $(u,v)\in \cH_k\oplus \cH_k\otimes\ell^2$ its orthogonal projection onto $\cM$, $P_\cM(u,v)$ is given by
\end{enumerate}
\vspace{-1em}\begin{equation}\label{eq:proj}
	 P_\cM(u,v) = (bw, w), \qquad w = M_b^*u + \Delta^2 v.
\end{equation}
\end{proposition}
\begin{proof} 
(i) If $((bf_n,f_n))$ converges in $\cH_k\oplus\cL_\Delta$ then from \[
	\|(bf_n,f_n)-(bf_m,f_m)\|^2
	=\|b(f_n-f_m)\|^2_{\cH_k} + \|\Delta(f_n-f_m)\|^2_{\cH_k\otimes\ell^2}
	= \|f_n-f_m\|^2_{\cH_k\otimes\ell^2},
\]
we see that $(f_n)$ is Cauchy in $\cH_k\otimes\ell^2$, hence convergent.

(ii) With $f\in \cH_k\otimes \ell^2$ and $w$ as in the statement we have	
	 \begin{align*}
	\la (bf, f), (bw, w)\ra_{\cH_k\oplus\cL_\Delta}
	&= \la f, w\ra_{ \cH_k\otimes\ell^2}
	\\	&= \la f, M_b^*u + \Delta^2 v\ra_{ \cH_k\otimes\ell^2}
	\\	&= \la (bf, f), (u, v)\ra{\cH_k\oplus\cL_\Delta}.
	\end{align*}
\end{proof}
The main result of this subsection is given below.

\begin{theorem}\label{thm:Hkb-unitary}
Suppose that $\Delta$ is  injective and let 
$\cM\subset \cH_k\oplus\cL_\Delta$ be the  subspace given by \eqref{M_def}.  
Then the map $J:\cM^\perp \to \cH_k(b)$, $J(u,v)=u$ is unitary. Equivalently, 
for every $f\in\cH_k(b)$  there exists a unique $g\in\cL_\Delta$ such that 
$(f,g)\in \cM^\perp$ and \[
	\|f\|_{\cH_k(b)}^2 = \|f\|_{\cH_k}^2+\|g\|_{\cL_\Delta}^2.
\]
\end{theorem}
\begin{proof}
Let  us note first  that   if $(0,g)\in\cM^\perp,~g\in \cL_\Delta$, then 
$g=0$. Indeed, in this case $g$ is orthogonal in $\cL_\Delta$ to the dense 
subspace $\cH_k\otimes \ell^2$, hence $g=0$. In particular, if 
$P_{\cM^\perp}=I-P_\cM$, this implies that the linear span of 
$\{P_{\cM^\perp}(k_y,0):~y\in \cX\}$ is dense in $\cM^\perp$.
By \cref{M_perp} (ii) we can easily  calculate the elements of this set. 
Indeed, according to \eqref{eq:proj} we have \[
	P_\cM(k_y,0) = (bw, w), \qquad w = M_b^*k_y,
\]
so that \[
	(I-P_\cM)(k_y,0)
	= P_{\cM^\perp}( k_y,0)
	= (k_y-bb^*(y)k_y, -b^*(y)k_y)
	= (k_y^b,-b^*(y)k_y),
\]
where $k^b$ denotes the reproducing kernel in $\cH_k(b)$. But then
\begin{align*}
	&\la P_{\cM^\perp}(k_y,0),P_{\cM^\perp}(k_z,0)\ra_{\cH_k\oplus\cL_\Delta}
	\\&=\la k_y-bb(y)^*k_y, k_z-bb(z)^*k_z\ra_{\cH_k}+\la\Delta^2 b(y)^*k_y,b(z)^*k_z\ra_{\cH_k\otimes\ell^2}
	\\&= k_y(z)-b(z)b(y)^*k_y(z)-\overline{b(y)b(z)^*k_z(y)}
	+\la bb(y)^*k_y, bb(z)^*k_z\ra_{\cH_k}
	\\&\quad + b(z)b(y)^*k_y(z)-\la bb(y)^*k_y, bb(z)^*k_z\ra_{\cH_k}
	\\&= k_y(z)-b(z)b(y)^*k_y(z)=k_y^b(z).
\end{align*}
Thus the map $J$ in the statement preserves the norm on a dense subset of 
$\cM^\perp$ and maps it onto a dense subset of $\cH_k(b)$. The result follows.
	
\end{proof}
The hypotheses that $\Delta$ is injective might seem restrictive. It does not 
apply to inner functions $b$ in the case of the Hardy space on $\bD$, or, 
more generally to sub-inner multipliers of $\cH_k$. It turns out that smaller 
spaces than $H^2$ on $\bD$, like for example standard weighted Dirichlet 
spaces possess a large set of  sub-inner multipliers 
(see \cite{aleman_freeouter}, Theorem 14.9). On the other hand, recall that 
weighted Bergman spaces do not have such multipliers.

\section{Approximation results}
\subsection{The general approximation theorem} Given a space $\cH_k(b)$,  one 
of the main difficulties in understanding its structure is the lack of 
knowledge about its elements other than the reproducing kernels. Therefore 
it is natural to relate this problem to elements in the larger (''known'') 
space $\cH_k$, in particular to the original kernels $k_y, ~y\in X$. We 
shall consider  the map $J$ from Theorem \ref{thm:Hkb-unitary}, more precisely, 
its inverse $J^{-1}:\cH_k(b)\to \cH_k\oplus\cL_\Delta$ and denote, as in the 
previous subsection, by $\cM$ the closed subspace of $\cH_k\oplus\cL_\Delta$ 
with \[
	\cM=\{(bg,g):~g\in \cH_k\otimes \ell^2\}.
\]
\begin{proposition}\label{kernelsin space} 
Assume that $\Delta$ is injective. If $y\in \cX$ then $k_y\in \cH_k(b)$ if 
and only if there exists $c_y>0$ such that \begin{equation}\label{eq:kernelmate1}
	|b(y)g(y)|\le c_y\|g\|_\Delta,\quad g\in \cH_k\otimes \ell^2.
\end{equation}
In this case, there exists a unique $l_y\in \cL_\Delta$ with
\begin{equation}\label{eq:kernelmate}
	\la g, l_y\ra_{\cL_\Delta} = -b(y)g(y), \quad  g\in\cH_k\otimes \ell^2,
\end{equation}
and $J^{-1} k_y=(k_y,l_y)$.
\end{proposition}
\begin{proof}
If $k_y\in \cH_k(b)$ and  $J^{-1} k_y=(k_y,l_y)\in \cM^\perp$, then $l_y$ 
must satisfy \eqref{eq:kernelmate} which implies \eqref{eq:kernelmate1}. 
Conversely, if \eqref{eq:kernelmate1} holds, the existence and uniqueness of $
l_y$ follows by the Riesz representation theorem and Theorem 
\ref{thm:Hkb-unitary} gives $k_y\in \cH_k(b)$.
\end{proof}
Under the assumption $k_y\in \cH(b),~y\in \cX$, the natural question 
regarding approximation in this space is whether the linear span of these 
functions is dense. We provide a sufficient condition  for this conclusion.
\begin{theorem}\label{thm:Hkb-k-dense}
Suppose that $\Delta$ is injective and that:\begin{enumerate}[label=(\roman*)]
\item If $y\in \cX,~e\in \ell^2$ there exists $d_y>0$ such that 
\end{enumerate}
\vspace{-0.5em}
      \begin{equation}\label{ev_Ldelta1}
      	|\la g(y),e\ra_{\ell^2}|\le d_y\|g\|_\Delta\|e\|_2,\quad g\in \cH_k\otimes \ell^2.
      \end{equation}
      Equivalently, point evaluations are continuous on $\cL_\Delta$.
\begin{enumerate}[label=(\roman*), resume]
\item If $f\in \cL_\Delta$ and $bf\in \cH_k$ then $f\in \cH_k\otimes \ell^2.$
\end{enumerate}
Then the linear span of $\{k_y:y\in\cX\}$ is  contained and  dense in $\cH_k(b)$.
\end{theorem}
\begin{proof} 
The hypothesis (i) together with \cref{kernelsin space} imply that 
$k_y\in \cH_k(b),\,y\in \cX$. Now suppose that $f\in\cH_k(b)$ annihilates 
$k_y,~y\in \cX$. Let \[
	J^{-1}f=(f,g), \,\, J^{-1}k_y=(k_y,l_y)\in \cM^\perp,\quad y\in \cX,
\]
and apply Proposition \ref{kernelsin space} to conclude that 
\eqref{eq:kernelmate} holds for all $y\in \cX$. Thus \[ 
	0
	= \la f,k_y\ra_{\cH_k(b)}
	= \la J^{-1}f,J^{-1}k_y\ra_{\cH_k\oplus\cL_\Delta}
	= f(y)-b(y)g(y),
\]
for all $y\in \cX$, i.e. $f=bg\in \cH_k$. But then by assumption  (ii), we 
have  $g\in \cH_k\otimes\ell^2$ and $J^{-1}f=(f,g)=(bg,g)\in \cM$ which 
implies $(f,g)=0$ and the result follows.
\end{proof}
\begin{remark}\label{basic assumptions}
\phantom{}
\begin{enumerate}[label=\arabic*), leftmargin=*, nolistsep]
\item	In the classical case when $\cH_k$ equals the Hardy space $H^2$ and $b$ 
	is a scalar analytic function bounded by $1$ in $\bD$, assumption (i) is 
	equivalent to the fact that $b$ is not extremal in $H^\infty$. This follows 
	immediately by the standard Szeg\"o theorem. Moreover, in this case there 
	exists an outer function $a$ with $|a|^2+|b|^2=1$ a.e. on the unit circle, 
	which easily implies that (ii) holds automatically.
\item	In the general case, an analogue of the above condition is to assume 
	that there exists a multiplier $a$ such that $I-M_b^*M_b\ge M_a^*M_a$ 
	(see for example \cite{jury-martin_extremal} and the discussion in the 
	Introduction). If $a$ is zero-free then this assumption implies 
	assumption (i) in the theorem above. However, we do not know whether 
	it implies (ii) as well.
\end{enumerate}\end{remark}

\section{Weighted sub-Bergman spaces}
Let $\Omega$ be a domain of $\bC^d$, $\mu$ be a positive measure on $\Omega$ 
such that the Bergman space $L^2_a(\mu)$ is a closed subspace of $L^2(\mu)$ 
with reproducing kernel $k^\mu$ (see \cref{sec:weighted-bergman}). For an 
analytic row-contraction $b$ on $\Omega\subset\bC^d$ we consider the  sub-
Bergman space $\cH_{k^\mu}(b)$. The aim of this subsection is to show that 
the general result given in \cref{thm:Hkb-k-dense} does always apply to these 
spaces.
\begin{theorem}\label{thm-Bergman-k-dense}
Assume that $\mu$ satisfies \eqref{same Bergman} and
let $b$ be any analytic row-contraction on $\Omega$ with $\|b(z_0)\|<1$ for 
some $z_0\in\Omega$. Then the linear span of  $\{k^\mu_z:~z\in \Omega\}$ is 
dense in $\cH_{k^\mu}(b)$.
\end{theorem}
\begin{proof}
The statement is obvious when $b$ is a constant with $\|b\|<1$, so we may 
assume $b$ is not constant. We want to apply  Theorem \ref{thm:Hkb-k-dense}. 
To this end we need to verify the assumptions (i) and (ii) in that theorem. 
Assumption (i) follows by  an application of Proposition \ref{pointev}. 
To verify (ii), let $g\in \cL_\Delta$ with $bg\in L_a^2(\mu)$ and let $(g_n)$ 
be a sequence in $L_a^2(\mu)\otimes\ell^2$ with $\|g_n-g\|_\Delta\to 0$. By 
assumption (i) we have for each $z\in \Omega$, 
$\lim_{n\to\infty}bg_n(z)=bg(z)$, and that $(g_n(z))$ converges weakly to 
$g(z)$ in $\ell^2$, in particular, $\|g(z)\|_2\le \liminf_{n\to\infty}\|g_n(z)\|_2$.
Then since $b(z)$ is a contractive row contraction we can apply Fatou's lemma 
to obtain \begin{align*}
	&\|g\|_\Delta^2=\lim_{n\to\infty}\|g_n\|_\Delta^2=\lim_{n\to\infty}\int_{\Omega}(\|g_n(z)\|_2^2-|bg_n(z)|^2)d\mu(z)
	\\&\qquad\ge\int_{\Omega}\liminf_{n\to\infty}(\|g_n(z)\|_2^2-|bg_n(z)|^2)d\mu(z)\ge  \int_{\Omega}
	 (\|g(z)\|_2^2-|bg(z)|^2)d\mu(z).
\end{align*}
This gives \[
	\int_{\Omega}\|g(z)\|^2d\mu(z)
	\le \|g\|_\Delta^2+\int_{\Omega}|bg(z)|^2d\mu(z),
\]
that is, assumption (ii) is verified as well. Then the result follows by a 
direct application of Theorem \ref{thm:Hkb-k-dense}.
\end{proof}

\subsection{Approximation by polynomials in standard weighted sub-Bergman spaces on the unit ball}
In this subsection we shall specialize to unitarily invariant weighted 
Bergman spaces on $\bB_d$, the unit ball in $\bC^d$. These are the spaces 
$\cH_\beta$ defined in \S 2.2, where $\beta>d$ corresponding to measures
that satisfy \eqref{same Bergman}. Recall that their reproducing kernel is 
given by \[
	s^\beta(z,w)=\frac1{(1-\la z,w\ra)^\beta}, \quad z,w\in \bB_d.
\]
For such $\beta$ we let  $b$  be an analytic row-contraction with $\|b(0)\|<1$.
The sub-Bergman space $\Hbeta(b)$ is the Hilbert space $\cH_{s^\beta}(b)$, 
that is, the Hilbert space with reproducing kernel \[
	s^\beta(z,w)(1-b(z)b(w)^*) = \frac{1-b(z)b(w)^*}{(1-\la z,w\ra)^\beta}.
\]
Let us note 
first the following property of the multipliers of these spaces.
\begin{proposition}\label{beta:kernel-factored}
Let $\beta>d$ and let $b$ an analytic row-contraction with $\|b(0)\|<1$. Then 
for every $0<\gamma<\beta-d$, $\Mult(\cH_\gamma)$ is contractively contained 
in $\Mult(\cH_\beta(b))$. In particular, if $u$ is analytic in a neighborhood 
of $\overline{\bB_d}$ then $u \in \Mult(\cH_\beta(b))$.
\end{proposition}
\begin{proof} 
The first assertion is a direct application of Proposition 
\ref{prop:kernel-factored} (i). The fact that analytic functions in a 
neighborhood of $\overline{\bB_d}$  are multipliers of 
$\cH_\gamma,~\gamma>0$, is well known and easy to prove. We omit the details.
\end{proof}
With this observation in hand we can turn to the approximation result for 
these spaces. It actually shows that within this context Sarason's dichotomy 
does not occur.
\begin{corollary}\label{polynomial-approx-Bergman}
Let $\beta>d$ and let $b$ an analytic row-contraction with $\|b(0)\|<1$. Then 
polynomials form a dense subset of $\cH_\beta(b)$. Moreover, for any $e\in\ell^2$,
the function $z\mapsto b(z)e$ belongs to $\Hbeta(b)$.
\end{corollary}
\begin{proof}
By Theorem \ref{thm-Bergman-k-dense}, $\cH_\beta(b)$ contains all kernels 
$s_w^\beta,~w\in \bB_d$, of $\cH_\beta$. By \cref{beta:kernel-factored} we 
have  $s_w^{-\beta}\in \Mult(\cH_\beta(b))$, hence 
$1=s_w^{-\beta}s_w^\beta\in \cH_\beta(b)$. The same argument shows that 
$1\in\cH_{\beta-\gamma}(b)$ whenever $\gamma>0,~ \beta-\gamma>d$. But then 
\cref{prop:kernel-factored} (ii) implies that for such $\gamma$, $\cH_\gamma$ 
is continuously contained in $\cH_\beta(b)$. Since the kernels $s_w^\beta$ 
can be approximated by polynomials in $\cH_\gamma$ the same holds in 
$\cH_\beta(b)$ and another application of Theorem \ref{thm-Bergman-k-dense} 
shows that the closure of polynomials in $\cH_\beta(b)$  contains the dense 
set $\text{span}\{s_w^\beta:~w\in \bB_d\}$ which concludes the proof.

The previous paragraph shows that $z\mapsto b(z)b(w)^*$ belongs to $\Hbeta(b)$
for any $w\in\Bd$. Indeed, this follows by multiplying the kernel function \[
	z\mapsto s^\beta_w(1-b(z)b(w)^*)(z)
\]
by $s^{-\beta}_w$, this proves that $z\mapsto 1-b(z)b(w)^*$ belongs to $\Hbeta(b)$,
but so does the constant function $1$. To conclude, let $e\in\ell^2$ be decomposed
as $e = e' + e''$ where $e'\in\Span\{b(w)^*:w\in\bB_d\}$ and $e''\perp b(w)^*$ for
all $w\in\bB_d$. Note that $e\perp b(w)^*$ is equivalent to $b(w)e=0$, hence the
function $z\mapsto b(z)e = b(z)e'$ belongs to $\Hbeta(b)$.
\end{proof}
Using a completely different method, Chu \cite{chu_density} proved the 
density of polynomials for scalar-valued $b$ in the case when $d=1$ and 
$\beta=2$.

\section{Approximation by polynomials in higher order de Branges-Rovnyak spaces}
Let $k$ be a reproducing kernel on a non-empty set $\cX$, and let $b$ be a 
contractive multiplier of $\cH_k$. In this section we shall consider kernels 
of the form \begin{equation}\label{high-order-ker}
	k^{b,m}(x,y)=k(x,y)(1-b(x)\overline{b(y)})^m,\quad x,y\in \cX,
\end{equation}
where $m$ is a fixed positive integer. Such kernels have been recently 
introduced and thoroughly studied in \cite{gu_hypercontractions}, in 
connection with $m-$hypercontractive operators.

\begin{definition}
A bounded operator $A$ on the Hilbert space $\cH$ is called  an 
\emph{$m$-hypercontraction} if \begin{equation}\label{eq:hypcontrac}
	\sum_{j=0}^n (-1)^j\binom njA^{*j}A^j \geq 0,
\end{equation}
for all $0\leq n\leq m$, where $m,n\in\bN$.
\end{definition}
The result below is a slight simplification of \cite[Corollary 3.6]{gu_hypercontractions}.

\begin{proposition}
Let $b$ be a contractive multiplier of $\cH_k$ such that $|b(x)|<1$ for every
$x\in\cX$. Then $M_b^*$ is a $m$-hypercontraction if and only if 
\begin{equation}\label{eq:b-hypcontrac-improved}
	k^{b,m}(x,y)=k(x,y)(1-b(x)\overline{b(y)})^m\gg 0.
\end{equation}
In this case, \[
	k(x,y)(1-b(x)b(y))^n\gg 0,
\]
for every integer $n$ with $0\le n\le m$.
\end{proposition}\begin{proof} 
We use repeatedly that the bounded operator $T$ on the reproducing kernel 
Hilbert space $\cH_k$ is positive  if and only if $Tk(x,y)\gg 0$. Thus if 
$M_b^*$ is $m-$hypercontractive then \eqref{eq:hypcontrac} holds for $n=m$, 
hence \begin{align*}	
	\sum_{j=0}^m (-1)^j\binom njM_b^jM_b^{*j}k(x,y)
	&=\sum_{j=0}^m( (-1)^j\binom njb(x)^j\overline{b(y)}^j)k(x,y)
	\\
	&=(1-b(x)\overline{b(y)})^mk(x,y)\gg 0.
\end{align*}
Conversely, if 	\eqref{eq:b-hypcontrac-improved} holds then 
\eqref{eq:hypcontrac} holds for $n=m$. Moreover, by Schur's product theorem any 
positive power of a complete Nevanlinna-Pick kernel is positive, therefore \[
	\frac{1}{(1-b(x)\overline{b(y)})^{m-n}}\gg 0,
\]
whenever $n<m$. Then another application of Schur's theorem gives that the 
product of the last two positive definite functions is positive definite, i.e. \[
	k(x,y)(1-b(x)b(y))^n\gg 0,
\]
whenever $n<m$. Equivalently, \eqref{eq:hypcontrac} holds for $0\le n<m$. 
Thus $M_b^*$ is $m-$hypercontractive and the result follows.
\end{proof}
We denote by $\cH_k(b,m)$ the Hilbert space with reproducing kernel 
$k^{b,m}$. We shall investigate  polynomial approximation in  $k=s^\beta$. To 
simplify notations we denote these spaces by $\cH_\beta(b,m)$. The following 
result completely solves a conjecture stated in 
\cite{gu_hypercontractions} in one complex variable ($d=1$).
\begin{corollary}\label{polynomial-approx-higher}
Let $d, m\ge 1$ be integers with $md<\beta$, let $k=s^\beta$   and  $b\in H^\infty(\bB_d)$ with 
$\|b\|_\infty\leq1$ and $|b(0)|<1$.  Then polynomials form a dense subset of $\cH_\beta(b,m)$.
\end{corollary}
\begin{proof}
By \cref{polynomial-approx-Bergman}, polynomials form a dense subset 
in $\cH_{\beta/m}(b)$, hence by \cref{prop:dense-algebra} they form a 
dense subset of $\cH_\beta(b,m)$.	
\end{proof}	
	
\begin{remark} 
The result holds when $d=1$ and $m=\beta$, provided that $b$ is a non-
extremal point in the unit ball of $H^\infty$. Instead of 
\cref{polynomial-approx-Bergman} just apply Sarason's dichotomy \cite{sarason} 
mentioned in the Introduction.

These methods generalize to cover other situations. For example, letting 
$b_1,\ldots,b_m$ be non-constant in the unit ball of $H^\infty$
one can prove  similar results about the space with kernel \[
	\frac{\prod_{j=1}^m(1-b_j(z)\overline{b_j(w)})}{(1-z\overline w)^\beta},
\]
where $\beta\geq m$. 

Furthermore, we can also consider the case when $b$ is vector-valued. However, 
in this case, when  $\beta=m$ Sarason's dichotomy is only known to 
hold when $b$ is a finite dimensional vector.
 
When $\beta>m$ all results mentioned in this remark extend to the context of 
several complex variables.
\end{remark}

\begin{question}
If $d=1$, $\beta=m$ and $b(z)=z$, then clearly $\cH_\beta(b,m)=\bC$. Does 
there exist $m>1$ and $b$ extremal in the unit  ball of $H^\infty$ such that 
$\cH_\beta(b,m)$ contains all polynomials? Do they form a dense subset?
\end{question}

\clearpage
\addcontentsline{toc}{section}{References}
\printbibliography
\end{document}